\documentclass{article}
\usepackage[utf8]{inputenc}
\usepackage{amssymb,latexsym}
\usepackage{amsmath, standalone}
\usepackage{graphicx}
\usepackage{textcomp}

\usepackage{pgfplots}

\usepackage{amsthm,amssymb,enumerate,graphicx, tikz}
\usepackage{amscd}
\usepackage{setspace}
\usepackage{comment, parskip}
\usepackage{hyperref}
\usepackage{xcolor}
\usepackage[capitalise]{cleveref}
\usepackage{subcaption}
\usepackage{enumerate}

\usepackage[left=1in,right=1in,top=1in,bottom=1in,bindingoffset=0cm]{geometry}

\newtheorem{theorem}{Theorem}[section]

\newtheorem*{theorem*}{Theorem}

\newtheorem{lemma}[theorem]{Lemma}

\newtheorem{question}[theorem]{Question}

\newtheorem{conjecture}[theorem]{Conjecture}
\theoremstyle{definition}

\theoremstyle{remark}

\newcommand{\R}{\mathcal{R}}

\newcommand{\I}{\mathcal{I}}

\newcommand{\B}{\mathcal{B}}

\newcommand*{\eqdef}{\stackrel{\mbox{\normalfont\tiny def}}{=}}
\newcommand{\abs}[1]{\lvert #1 \rvert}
\DeclareMathOperator{\disc}{disc}
\newcommand{\mcal}[1]{\mathcal{#1}}

\title{A finite victory over de Bruijn--Erd\H{o}s in interval discrepancy}
\author{%
Jared DeLeo\thanks{Department of Mathematics and Computer Science, LaSalle University, Philadelphia, PA, USA. \texttt{deleoj@lasalle.edu}}
\and
Owen Henderschedt\thanks{Department of Mathematics and Statistics, Auburn University, Auburn, AL, USA. \texttt{olh0011@auburn.edu}}
\and
Chris Wells\thanks{Department of Mathematics and Statistics, Haverford College, Haverford, PA, USA. \texttt{cwells@haverford.edu}.}}

\date{}

\begin{document}

\maketitle

\begin{abstract}
We study a finite form of the classical interval discrepancy problem. Starting from the unit interval, one repeatedly splits an existing interval into two until $n$ intervals have been produced.
The discrepancy of such a process is the maximum, over all intermediate stages, of the ratio between the longest interval and the shortest interval.
A theorem of de Bruijn and Erd\H{o}s from 1949 shows that this ratio must approach $2$ as $n\to\infty$, and they give a sharp construction achieving this bound.
For fixed $n$, their construction gives the upper bound $\disc(n)\leq 2-\frac{3}{2n}+O\bigl(\frac 1{n^2}\bigr)$.
In this paper, we prove that $\disc(n)=2^{1-1/\lceil n/2\rceil}=2-\frac{4\ln 2}{n}+O\bigl(\frac 1{n^2}\bigr)$ for every $n$.
\end{abstract}

\section{Introduction}

How should we place a set of points in a geometric domain so that the deviation from a uniform distribution is minimized?
This deviation is known as the \emph{discrepancy}, and the precise meaning of ``deviation from uniformity'' depends on the setting.
For example, if the domain is the unit interval $[0,1]$ and we wish to place a fixed finite number of points, it is natural to expect that an arrangement minimizing discrepancy is obtained by placing the points at positions ${i\over n+1}$ for $i\in\{1,\dots,n\}$.
Quoting Ji\v{r}\'i Matou\v{s}ek~\cite{GeoDiscrepancy}: ``[Such an arrangement] hardly finds serious competitors as a candidate for the most uniformly distributed $n$-point set in the unit interval''.
Of course, the problem is no longer so trivial when the domain is more complicated, when the number of points is infinite, or when the points must be chosen sequentially; such problems give rise to the area of \emph{geometric discrepancy theory} (see the book~\cite{GeoDiscrepancy}).

A classical way to measure how evenly points divide an interval or circle is to compare the longest gap with the shortest gap. For a fixed number of points in $[0,1]$, the completely balanced configuration is again optimal, since all resulting subintervals have equal length. The situation changes dramatically when the points are placed one at a time and one asks how well-balanced the gaps can remain throughout the entire process.

This sequential problem has a classical history. In 1949, de Bruijn and Erd\H{o}s~\cite{deb} studied infinite sequences of points placed one at a time on the circle $S^1$. They proved that, for every such sequence, the ratio of the largest gap to the smallest gap must be arbitrarily close to $2$ infinitely often. Moreover, this is sharp: they constructed an extremal sequence, given by
\[
    x_k=\log_2(2k-1)\pmod 1,
\]
showing that the asymptotic constant $2$ cannot be improved. After the first point is placed on the circle, one may cut the circle at that point and unwrap it into the unit interval; the gaps on the circle then become the intervals of a partition of $[0,1]$. Thus, in studying the ratio between the longest and shortest subintervals, the circle and interval formulations are equivalent.

This raises a natural question: although the de Bruijn--Erd\H{o}s construction is asymptotically sharp, can it be improved when we only care about the process up to a fixed number of intervals? More precisely, given a positive integer $n$, we ask for the best possible strategy up to the moment when $n$ intervals have been produced. We formalize this now.

Instead of considering the points themselves, it will be convenient to work directly with the induced interval partitions. For a partition $\mcal I$ of $[0,1]$ into intervals, define
\[
    \disc(\mcal I)\eqdef\max_{I_1,I_2\in\mcal I}{\abs{I_1}\over\abs{I_2}},
\]
where $\abs{I}$ denotes the length of the interval $I$.

A \emph{strategy of length $n$} is a sequence $(\mcal I_1,\mcal I_2,\dots,\mcal I_n)$ such that $\mcal I_i$ is a partition of $[0,1]$ into $i$ intervals and $\mcal I_{i+1}$ is formed from $\mcal I_i$ by dividing a single interval into two. For a strategy $\mcal S=(\mcal I_1,\mcal I_2,\ldots,\mcal I_n)$, define
\[
    \disc(\mcal S)\eqdef\max_{t\in[n]}\ \disc(\mcal I_t).
\]
Finally, define
\[
    \disc(n)\eqdef\min_{\mcal S}\ \disc(\mcal S),
\]
where the minimum is taken over all strategies $\mcal S$ of length $n$.

Observe that $\disc(1)=\disc(2)=1$.
Already, determining $\disc(3)$ is non-trivial, but can be found without too much effort.
First, notice that for any strategy $\mathcal{S}=(\mcal I_1,\mcal I_2,\mcal I_3)$ of length $3$ that realizes $\disc(3)$, we may assume $\mathcal{I}_3$ is obtained from $\mathcal{I}_2$ by splitting the largest interval of $\mathcal{I}_2$ exactly in half.
Thus, by employing symmetry, we may assume that there exists some $x\in [1/2,1]$ such that any strategy of length $3$ realizing $\disc(3)$ looks like
\[
    \bigl(\bigl\{[0,1]\bigr\}, \bigl\{[0,x], [x,1]\bigr\}, \bigl\{[0, x/2], [x/2, x], [x,1]\bigr\}\bigr).
\]
Therefore,
\[
    \disc(3) = \min_{x\in [1/2,1]}\max\biggl\{\frac{x}{1-x}, \frac{1-x}{x/2}\biggr\}.
\]
Basic calculus indicates that the minimum occurs at $x = 2-\sqrt2$; therefore $\disc(3) = \sqrt{2}\approx 1.4142$.

With this notation, the theorem of de Bruijn and Erd\H{o}s implies that
$\disc(n)\to 2$ as $n\to\infty$.  The base-$2$ logarithmic sequence of de Bruijn and Erd\H{o}s induces a strategy of length $n$, giving the following upper bound.

\begin{theorem}[de Bruijn--Erd\H{o}s~\cite{deb}]\label{thm:dBE-finite}
    For every positive integer $n$,
    \[
        \disc(n)
        \leq
        \frac{\log(1+1/n)}
        {\log\bigl((1-\frac1{2n})^{-1}\bigr)}
        =
        2-\frac{3}{2n}+O\biggl(\frac1{n^2}\biggr).
    \]
\end{theorem}

The main result of this paper is a complete determination of $\disc(n)$:

\begin{theorem}\label[theorem]{thm:lex-merge}
    For every positive integer $n$,
    \[
        \disc(n)=2^{1-1/\lceil n/2\rceil}
        =
        2-\frac{4\ln 2}{n}+O\biggl(\frac1{n^2}\biggr).
    \]
\end{theorem}

It is important to remark that in the original version of this paper, we proved the upper-bound of \Cref{thm:lex-merge} but were only able to establish a lower-bound of $\disc(n)\geq 2-\frac{6\ln 2}{n}-O\bigl(\frac1{n^2}\bigr)$.
Shortly after we posted the original version of this paper, Jaan Allikvere contacted us and shared an AI-assisted proof of a matching lower-bound.~\cite{allikvere}
With his permission, we include this matching lower-bound as part of this paper, especially since the proof is arguably simpler than that of our original, worse bound.

%
%
%

This problem and its variants have also appeared under the terminology of \emph{balanced stick breaking}. We take a brief moment to discuss a few of these variants. In their original paper, de Bruijn and Erd\H{o}s~\cite{deb} considered a generalization in which one compares the largest total length of $r$ consecutive subintervals with the smallest total length of $r$ consecutive subintervals. Recent progress on this problem was made by Clément and Steinerberger~\cite{clement2025balancedstickbreaking}, who showed that this ratio, in the limit, is at most $1+O(\log r/r)$. Various studies have also explored the optimal spacing of points on a circle, such as angles between close-to-evenly distributed points (Ravenstein, \cite{vanRavenstein1989OptimalSpacing}) and number of points within a finite distance along the circle (Ramshaw, \cite{Ramshaw1978GapStructure}). A survey of balanced stick breaking, related distributions of points on a circle, and analysis on interval sizes is provided by Habib \cite{Habib2014EquitableDistribution}.

We now outline the rest of the paper. In \cref{sec:lower}, we prove the lower-bound of \cref{thm:lex-merge}.
In \cref{sec:lex-merge}, we introduce the strategy used to prove upper-bound of \cref{thm:lex-merge}, which we call \emph{lex-merge}.
In \cref{sec:structure}, we establish several structural properties of lex-merge.
Then, in \cref{sec:upper}, we use these properties to actually prove the upper-bound of \cref{thm:lex-merge}.
Finally in \cref{sec:conclusion} we conclude with some open problems.

\paragraph{Acknowledgments.}
We thank Joseph Briggs for introducing us to this problem and for helpful discussions throughout our work on it.
We are also thankful to Jaan Allikvere for sharing his AI-assisted proof of the lower-bound of \Cref{thm:lex-merge} with us and for allowing us to include it in this paper.

\section{Proof of the lower-bound of \Cref{thm:lex-merge}}\label{sec:lower}
As mentioned in the introduction, the following proof was communicated to us by Jaan Allikvere~\cite{allikvere} and was constructed with the help of the Cursor AI model.

\begin{lemma}\label[lemma]{lem-lowerbound}
    Suppose $\mcal S = (\mcal I_1, \dots, \mcal I_n)$ is a strategy.
    If $m_k$ denotes the length of the smallest interval of $\mcal I_k$, then
    \[
        m_{k+1}\leq\frac{\disc(\mcal S)}2 m_k,\qquad\text{for each }k\in[n-1].
    \]
\end{lemma}
\begin{proof}
    Let $M_k$ denote the length of the largest interval in $\mcal I_k$.
    In the transition from $\mcal I_k$ to $\mcal I_{k+1}$, we split some interval $I\in\mcal I_k$ into two intervals $I_1,I_2\in\mcal I_{k+1}$.
    Therefore,
    \[
        m_{k+1}\leq\min\{\abs{I_1},\abs{I_2}\}\leq{\abs I\over 2}\leq \frac{M_k}2.
    \]
    Since $M_k/m_k\leq\disc(\mcal S)$, the claim follows.
\end{proof}

\begin{theorem}
    For every positive integer $n$,
    \[
        \disc(n)\geq 2^{1-1/\lceil n/2\rceil}.
    \]
\end{theorem}
\begin{proof}
    Suppose $\mcal S = (\mcal I_1, \dots, \mcal I_n)$ is a strategy, set $D=\disc(\mcal S)$, fix $k = n - \bigl\lceil \frac{n}{2}\bigr\rceil + 1$, and say $m_k$ is the length of the smallest interval in $\mcal I_k$.
    Then, since exactly one interval is changed in each iteration of $\mcal I_i$ to $\mcal I_{i+1}$ for each $i$, there is some $I\in\mcal I_k$ with $I\in\mcal I_n$ as well.
    In particular, using $M_n$ to denote the length of the largest interval in $\mcal I_n$,
    \[
        m_k \leq \abs I \leq M_n \leq Dm_n.
    \]
    Further, by Lemma \ref{lem-lowerbound}, iterating $m_{i+1} \leq \frac{D}{2}m_i$ a total of $\bigl\lceil\frac n 2\bigr\rceil - 1$ times yields
    \[
        m_n \leq \left(\frac{D}{2}\right)^{\lceil n/2 \rceil-1}m_k \leq \frac{D^{\lceil n/2 \rceil}}{2^{\lceil n/2 \rceil-1}}m_n \qquad\implies\qquad D \geq 2^{1-1/\lceil n/2\rceil}.
    \]
    Since $\mcal S$ was arbitrary we conclude that $\disc(n)\geq 2^{1-1/\lceil n/2\rceil}$.
\end{proof}

\section{The lex-merge strategy}\label{sec:lex-merge}
Instead of describing a strategy for starting with an interval and splitting it into $n$, we start with $n$ intervals and describe a strategy for merging them together into one.
These two viewpoints are equivalent upon simply reversing the process.
\medskip

Fix $n\geq 1$ and set
\[
m\eqdef \left\lceil \frac n2\right\rceil.
\]

A \textit{basket} $B$ is a multiset on $\{0,1,\dots,m-1\}$ such that the multiplicity of each element is at most $2$.

Given a basket $B$, the \textit{length} of $B$ is defined to be
\[
\ell(B)\eqdef\sum_{x\in B} 2^{x/m}.
\]

Let $\mathcal{B}$ be a collection of baskets. The \textit{discrepancy} of $\mathcal{B}$ is
\[
\disc(\mathcal{B})=\frac{\max\{\ell(B):B\in\mathcal{B}\}}{\min\{\ell(B):B\in\mathcal{B}\}}.
\]

When comparing baskets lexicographically, we always write their elements in weakly increasing order. Given baskets $B_1$ and $B_2$, we say that $B_1$ is \textit{less than $B_2$ lexicographically}, and write $B_1\prec B_2$, if either
\begin{enumerate}
    \item[(i)] $|B_1|<|B_2|$, or
    \item[(ii)] $|B_1|=|B_2|$ and $B_1$ precedes $B_2$ in the usual lexicographic ordering.
\end{enumerate}

\textbf{Notation.} We use $[x_1,x_2,\dots, x_t]$ to denote the basket containing the elements $x_1,\dots,x_t$. In this way, braces $\{\ \}$ denote a collection of baskets, while brackets $[\ ]$ denote a single basket.

We recursively define a sequence of collections of baskets
\[
\mathcal{B}_0,\mathcal{B}_1,\dots,\mathcal{B}_{n-1}.
\]
First, let
\[
\mathcal{B}_0=
\begin{cases}
\{[0],[0],[1],[1],\dots,[m-2],[m-2],[m-1]\}, & \text{if $n$ is odd},\\[0.3em]
\{[0],[0],[1],[1],\dots,[m-2],[m-2],[m-1],[m-1]\}, & \text{if $n$ is even}.
\end{cases}
\]

Now for the recursion. Suppose
\[
\mathcal{B}_i=\{B_1,\dots,B_t\},
\]
where
\[
B_1\preceq B_2\preceq \cdots \preceq B_t,
\]
with ties (i.e., identical baskets) broken arbitrarily.
We obtain $\mathcal{B}_{i+1}$ by merging $B_1$ and $B_2$, that is, by removing $B_1$ and $B_2$ and replacing them with their multiset union
$M\eqdef B_1\uplus B_2.$
Thus,
\[
\mathcal{B}_{i+1}=\{B_3,\dots,B_t,M\}.
\]
We then reorder the baskets of $\mathcal{B}_{i+1}$ so that they are again listed increasingly with respect to $\preceq$.

We call this algorithm generating $\mathcal{B}_0,\dots,\mathcal{B}_{n-1}$ the \textit{lex-merge strategy of order $n$}, and denote it by $\mathrm{LM}_n$.
For example $\mathrm{LM}_7$ looks like
\begin{align*}
    \mathcal B_0 &= \{[0],[0],[1],[1],[2],[2],[3]\},\\
    \mathcal B_1 &= \{[1],[1],[2],[2],[3],[0,0]\},\\
    \mathcal B_2 &= \{[2],[2],[3],[0,0],[1,1]\},\\
    \mathcal B_3 &= \{[3],[0,0],[1,1],[2,2]\},\\
    \mathcal B_4 &= \{[1,1],[2,2],[0,0,3]\},\\
    \mathcal B_5 &= \{[0,0,3],[1,1,2,2]\},\\
    \mathcal B_6 &= \{[0,0,1,1,2,2,3]\}.
\end{align*}
Let
\[
\disc(\mathrm{LM}_n)=\max\{\disc(\mathcal{B}_i): i\in\{0,1,\dots,n-1\}\}.
\]

We make two remarks about this strategy. First,
\[
    L \eqdef \sum_{B\in \mathcal{B}_i}\ell(B)
\]
is independent of \(i\), but need not equal \(1\). Therefore, to obtain a strategy \(\mathcal{S}\) on the unit interval (or $S^1$), we normalize each basket length by \(L\). Since discrepancy is invariant under scaling, this normalization does not change any of the ratios.

Second, as mentioned at the beginning of this section, lex-merge is described as a merging process instead of as a splitting process. This causes no difficulty: the desired strategy is obtained by running lex-merge in reverse. Thus, for each \(1\leq i\leq n-1\), the transition from \(\mathcal{B}_i\) to \(\mathcal{B}_{i-1}\) corresponds to splitting a single interval into two.

\subsection{Structural properties of the lex-merge strategy}\label{sec:structure}

We now prove several structural properties about the baskets that arise in the lex-merge strategy of order $n$.
For integers $a,h$ define the cyclic interval of length $h$ starting at $a$ to be
\[
    I_h(a)\eqdef [a,a,a+1,a+1,\dots,a+h-1,a+h-1],
\]
where
\begin{enumerate}
    \item The elements are taken modulo $m$, and
    \item The element $m-1$ appears only once if $n$ is odd. That is, for example, if $n=9$ then $I_4(2)=[2,2,3,3,4,0,0]$, whereas if $n=10$ then $I_4(2)=[2,2,3,3,4,4,0,0]$.
\end{enumerate}
To make this definition more transparent, note that
\[
    |I_h(a)|=\begin{cases}
        2h, & \text{if $a+h<m-1$ or $n$ is even},\\
        2h-1, & \text{if $a+h\geq m-1$ and $n$ is odd}.
    \end{cases}
\]

We say a basket $B$ is \textit{cyclically ordered} if either $|B|=1$ or $B=I_h(a)$ for some $a,h$.
If $B$ contains both $0$ and $m-1$, we say $B$ is \textit{wrapped}.
We say a collection of baskets $\mathcal{B}$ has \textit{size-3-symmetry} if there exists integers $r,w$ such that $2^r < w < 2^{r+1}$ and all baskets $B\in \mathcal{B}$ have size $2^r$ or $2^{r+1}$ except possibly one wrapped basket of size $w$.

Given a collection of baskets $\mathcal{B}$, we say $\mathcal{B}'\subseteq \mathcal{B}$ is a \textit{chain} if there exists an ordering $B_1, \dots B_t$ of $\mathcal{B}'$ such that $B_i=I_{h_i}(a_i)$ where $a_{i+1}=a_i+h_i$.

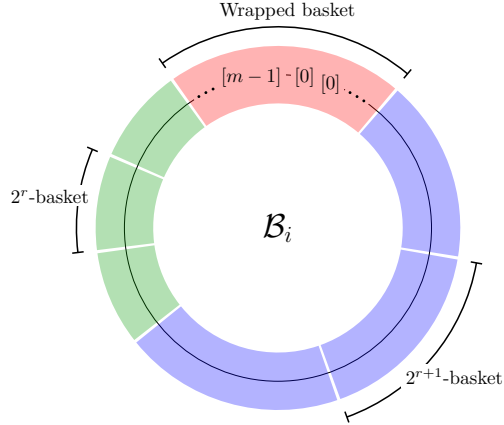
\begin{figure}[h!]
    \centering
    \begin{tikzpicture}

\def\R{80pt}
\def\O{95pt}
\def\I{65pt}

\draw[fill=white] (0,0) circle (\R);

\newcommand{\annulussector}[3]{%
    \path[
        fill=#1,
        fill opacity=0.3,
        draw=none
    ]
    (#2:\O)
    arc (#2:#3:\O)
    -- (#3:\I)
    arc (#3:#2:\I)
    -- cycle;
}

\annulussector{red}{50}{125}
\annulussector{blue}{49}{-9}
\annulussector{blue}{-10}{-70}
\annulussector{blue}{-71}{-141}
\annulussector{green!60!black}{126}{156}
\annulussector{green!60!black}{157}{187}
\annulussector{green!60!black}{188}{218}

\node[fill=red!30, inner sep=2pt, outer sep=0pt] at (90-10:\R) {$[0]$};
\node[fill=red!30, inner sep=2pt, outer sep=0pt] at (90-20:\R) {$[0]$};
\node[fill=red!30, inner sep=2pt, outer sep=0pt] at (90+10:\R) {$[m-1]$};

\foreach \a in {87+28,87+31,87+34} {
    \fill[red!30] (\a:\R) circle (3pt);
    \fill[black] (\a:\R) circle (.75pt);
}

\foreach \a in {90-34,90-31,90-28} {
    \fill[red!30] (\a:\R) circle (3pt);
    \fill[black] (\a:\R) circle (.75pt);
}

\def\B{105pt}

\draw[thick] (-10:\B) arc (-10:-70:\B);

\draw[thick] (-10:102pt) -- (-10:108pt);
\draw[thick] (-70:102pt) -- (-70:108pt);

\node[fill=white, inner sep=2pt] at (-40:120pt) {$2^{r+1}$-basket};

\draw[thick] (157:\B) arc (157:187:\B);

\draw[thick] (157:102pt) -- (157:108pt);
\draw[thick] (187:102pt) -- (187:108pt);

\node[fill=white, inner sep=2pt] at (172:120pt) {$2^{r}$-basket};

\draw[thick] (50:\B) arc (50:125:\B);

\draw[thick] (50:102pt) -- (50:108pt);
\draw[thick] (125:102pt) -- (125:108pt);

\node[fill=white, inner sep=2pt] at (87.5:113pt) {Wrapped basket};

\node at (0,0) {\LARGE{$\mathcal{B}_i$}};

\end{tikzpicture}
    \caption{A visual of properties (P1)-(P3) in Lemma \ref{lem:structural-props} for the collection $\mathcal{B}_i$}
    \label{fig:properties}
\end{figure}

\begin{lemma}\label{lem:structural-props}
Fix $n$, and let $0\leq i\leq n$. Let $\mathcal{B}_i$ be the $i$th collection of baskets in the lex-merge strategy of order $n$. Then the following properties hold.

\begin{enumerate}
    \item[(P1)] Every basket in $\mathcal{B}_i$ is cyclically ordered.
    \item[(P2)] The collection $\mathcal{B}_i$ has size-3-symmetry.
    \item[(P3)] If $\mathcal B'$ and $\mathcal B''$ are the collections of baskets of size $2^r$ and $2^{r+1}$, respectively, then both $\mathcal B'$ and $\mathcal B''$ form chains. Furthermore, $(\mathcal B'', \mathcal B')$ (in this order) also forms a chain.
\end{enumerate}
\end{lemma}

\begin{proof}
    We prove this by induction on $i$. By the definition of $\mathcal{B}_0$, the first $m$ merges consist of merging $[0],[0]\to [0,0]$, then $[1],[1]\to [1,1]$, through $[m-1],[m-1]\to [m-1,m-1]$ (so long as $n$ is even). All of these collections trivially satisfy (P1)-(P3). We may now assume that if $i\in B$, there are both copies of $i$ in $B$ for all $0\leq i\leq m-2$ (and $m-1$ if $n$ is even). We assume the statement holds through $\mathcal{B}_i$, and we now show that $\mathcal{B}_{i+1}$ satisfies the desired properties. Let
    \[
        \mathcal{B}_i = \{B^{(i)}_1,\dots , B_t^{(i)}\} \qquad \text{where} \qquad B^{(i)}_1\prec \dots \prec B^{(i)}_t,
    \]
    and let
    \[
        \mathcal{B}_{i+1} = \{B^{(i+1)}_1,\dots , B_{t-1}^{(i+1)}\} \qquad \text{where} \qquad B^{(i+1)}_1\prec \dots \prec B^{(i+1)}_{t-1}.
    \]
    Thus, $B^{(i+1)}_j=B^{(i)}_{j+2}$ for all $j\in [t-2]$, and $B^{(i+1)}_{t-1} = B^{(i)}_1 \uplus B_2^{(i)}$. Let $M = B^{(i)}_1 \uplus B_2^{(i)}$ be the merged basket to go from $\mathcal{B}_i$ to $\mathcal{B}_{i+1}$.

    \textbf{Proof that $\mathcal{B}_{i+1}$ satisfies (P1).} Every basket of $\mathcal{B}_{i+1}$ except $M$ already belongs to $\mathcal{B}_i$, and hence is cyclically ordered by the induction hypothesis. Thus, it remains only to show that $M$ is cyclically ordered.

    First, suppose that $|B^{(i)}_1|=|B^{(i)}_2|$. Since the lex-merge strategy merges the two lexicographically smallest baskets, these two baskets have minimum size among all baskets in $\mathcal{B}_i$. By the size-3-symmetry property, and by choosing $r$ so that $2^r$ is the minimum basket size appearing in $\mathcal{B}_i$, we may assume that
    $|B^{(i)}_1|=|B^{(i)}_2|=2^r$. By (P3), the baskets of size $2^r$ form a chain. Since $B^{(i)}_1$ and $B^{(i)}_2$ are the two lexicographically smallest baskets, they are consecutive baskets in this chain. Therefore, their multiset union $M=B^{(i)}_1\uplus B^{(i)}_2$ is again a cyclically ordered basket.

    Now suppose that $|B^{(i)}_1|<|B^{(i)}_2|$. Since the possible basket sizes in $\mathcal{B}_i$ are $2^r$, $w$, and $2^{r+1}$, the pair
    $\bigl(|B^{(i)}_1|,|B^{(i)}_2|\bigr)$ must fall into one of the following cases:
    \[
        (2^r,w),\qquad (w,2^{r+1}),\qquad \text{or}\qquad (2^r,2^{r+1}).
    \]
    We first consider the case
    $|B^{(i)}_1|=2^r$ and $|B^{(i)}_2|=w$. Since $|B^{(i)}_1|<|B^{(i)}_2|$, the basket $B^{(i)}_1$ is the unique basket of size $2^r$ in $\mathcal{B}_i$. Also, by size-3-symmetry, $B^{(i)}_2$ is the unique exceptional basket of size $w$, and it is wrapped. Thus, for some integers $a$ and $h$, we have $B^{(i)}_2=I_h(a)$.
    By (P3), the baskets of size $2^{r+1}$ followed by the baskets of size $2^r$ form a chain. Since $B^{(i)}_1$ is the only basket of size $2^r$, it must be the block immediately preceding the wrapped basket $B^{(i)}_2$. Therefore $B^{(i)}_1=I_{2^{r-1}}(a-2^{r-1})$.
    Hence, $M=I_{2^{r-1}}(a-2^{r-1})\uplus I_h(a)=I_{2^{r-1}+h}(a-2^{r-1})$ and is thus cyclically ordered in this case.

    Next, suppose that $|B^{(i)}_1|=w$ and $|B^{(i)}_2|=2^{r+1}$. Since $B^{(i)}_1$ has the exceptional size $w$, size-3-symmetry implies that $B^{(i)}_1$ is wrapped. Thus, for some integers $a$ and $h$, we may write $B^{(i)}_1=I_h(a)$.
    Since $B^{(i)}_1$ is the smallest basket and has size $w$, there are no baskets of size $2^r$ in $\mathcal{B}_i$.
    Hence, the next smallest basket $B^{(i)}_2$ has size $2^{r+1}$.
    By (P3), the baskets of size $2^{r+1}$ form a chain.
    Since $B^{(i)}_2$ is the lexicographically smallest basket of size $2^{r+1}$, it is the first block in this chain and thus is the block succeeding the wrapped basket $B_1^{(i)}$. Therefore $B^{(i)}_2=I_{2^r}(a+h)$.
    It follows that $M=I_{h+2^r}(a)$ is cyclically ordered in this case.

    Finally, suppose that $|B^{(i)}_1|=2^r$ and $|B^{(i)}_2|=2^{r+1}$. Since there is no basket of exceptional size $w$ between these two sizes, size-3-symmetry implies that there is no wrapped basket in $\mathcal{B}_i$. Also, since $|B^{(i)}_1|<|B^{(i)}_2|$, the basket $B^{(i)}_1$ is the unique basket of size $2^r$. By (P3), the baskets of size $2^{r+1}$ followed by the baskets of size $2^r$ form a chain. Since there is no wrapped basket, this chain runs from $0$ to $m-1$. Therefore,, the unique basket of size $2^r$ is the final block of this chain, while $B^{(i)}_2$ is the first block of this chain and thus is the block succeeding $B_1^{(i)}$. Hence $B^{(i)}_1=I_{2^{r-1}}(m-2^{r-1})$ and $B^{(i)}_2=I_{2^r}(0)$.
    Thus, $M=I_{2^{r-1}+2^{r}}(m-2^{r-1})$ is cyclically ordered in this case.

    \textbf{Proof that $\mathcal{B}_{i+1}$ satisfies (P2).}
    By the induction hypothesis, $\mathcal{B}_i$ has size-3-symmetry. Thus, there exist integers $r,w$ with $2^r < w < 2^{r+1}$ such that every basket in $\mathcal{B}_i$ has size $2^r$ or $2^{r+1}$, except possibly one wrapped basket of size $w$.

    First, suppose that $|B^{(i)}_1|=|B^{(i)}_2|$. Since there is at most one exceptional basket of size $w$, these two baskets cannot both have size $w$. Thus, after choosing $r$ so that $2^r$ is the minimum basket size that appears, we may assume
    $|B^{(i)}_1|=|B^{(i)}_2|=2^r$. Hence $|M|=2^{r+1}$. Therefore, $\mathcal{B}_{i+1}$ still has size-3-symmetry with respect to $2^r<w<2^{r+1}$.

    Now suppose that $|B^{(i)}_1|<|B^{(i)}_2|$. As in the proof of (P1), there are three cases.

    First, suppose that $|B^{(i)}_1|=2^r$ and $|B^{(i)}_2|=w$. Since $B^{(i)}_1$ is the unique basket of size $2^r$, after the merge, there are no baskets of size $2^r$. Moreover,
    $|M|=2^r+w$.
    Since $2^r<w< 2^{r+1}$, we have
    \[
        2^{r+1}< 2^r+w < 2^r+2^{r+1}<2^{r+2}.
    \]
    Hence $\mathcal{B}_{i+1}$ has size-3-symmetry with respect to $2^{r+1}<2^r+w<2^{r+2}$.

    Next, suppose that $|B^{(i)}_1|=w$ and $|B^{(i)}_2|=2^{r+1}$. Since $B^{(i)}_1$ is the smallest basket lexicographically, there are no baskets of size $2^r$ in $\mathcal{B}_i$. Thus, $
    |M|=w+2^{r+1}$.
    Since $2^r < w < 2^{r+1}$, this size satisfies
    \[
        2^{r+1}<w+2^{r+1} < 2^{r+2}.
    \]
    Therefore, $\mathcal{B}_{i+1}$ has size-3-symmetry with respect to $2^{r+1}<w+2^{r+1}<2^{r+2}$.

    Finally, suppose that $|B^{(i)}_1|=2^r$ and $|B^{(i)}_2|=2^{r+1}$. Since $B^{(i)}_1$ is the unique basket of size $2^r$ and there is no wrapped basket of size $w$, after the merge, no basket of size $2^r$ remains, we have $|M|=2^r+2^{r+1}=3\cdot 2^r$.
    Since
    \[
        2^{r+1}< 3\cdot 2^r < 2^{r+2},
    \] $\mathcal{B}_{i+1}$ has size-3-symmetry with respect to $2^{r+1}<3\cdot 2^r<2^{r+2}$.

    In all cases, $\mathcal{B}_{i+1}$ has size-3-symmetry. Thus, $\mathcal{B}_{i+1}$ satisfies (P2).

    \textbf{Proof that $\mathcal{B}_{i+1}$ satisfies (P3).} First suppose that $|B^{(i)}_1|=|B^{(i)}_2|$. As before, we may assume that these two baskets have size $2^r$. By (P3), the baskets of size $2^r$ form a chain. Since $B^{(i)}_1$ and $B^{(i)}_2$ are the two lexicographically smallest baskets of size $2^r$, they are the first two baskets in this chain. Thus, for some integer $a$, we may write $B^{(i)}_1=I_{2^{r}}(a)$ and $B^{(i)}_2=I_{2^r}(a+2^{r})$
    Therefore the merged basket is
    $M=B^{(i)}_1\uplus B^{(i)}_2=I_{2^{r+1}}(a)$,
    which has size $2^{r+1}$.

    We now check that the chain property is preserved. Removing $B^{(i)}_1$ and $B^{(i)}_2$ from the size-$2^r$ chain leaves a chain of baskets of size $2^r$, possibly empty. Moreover, by the combined chain property in (P3), the baskets of size $2^{r+1}$ followed by the baskets of size $2^r$ form a chain. Hence, if there is already a basket of size $2^{r+1}$ immediately preceding $B^{(i)}_1$, then the last such basket is
    $[a-2^r,a-2^r,\dots, a-1,a-1]$.
    Thus, inserting $M$ after the size-$2^{r+1}$ baskets extends that chain by one consecutive basket.

    Consequently, in $\mathcal{B}_{i+1}$, the baskets of size $2^{r+1}$ form a chain, the remaining baskets of size $2^r$ form a chain, and, in this order, these two subcollections still form a chain.

    Now suppose that $|B^{(i)}_1|<|B^{(i)}_2|$. As in the previous parts, there are three cases.

    First, suppose that $|B^{(i)}_1|=2^r$ and $|B^{(i)}_2|=w$. Since the sizes are strictly increasing, $B^{(i)}_1$ is the unique basket of size $2^r$. Also, by size-3-symmetry, $B^{(i)}_2$ is the unique exceptional wrapped basket of size $w$. Therefore, after merging $B_1^{(i)}$ and $B_2^{(i)}$, every remaining basket in $\mathcal{B}_i$ has size $2^{r+1}$, in which these baskets formed a chain by (P3). Since none of them are changed by the merge, they still form a chain in $\mathcal{B}_{i+1}$. The new basket $M$ is the unique exceptional wrapped basket, as shown in the proof of (P1), and hence is not part of the power-of-two chains. Thus $\mathcal{B}_{i+1}$ satisfies (P3) in this case.

    Next, suppose that $|B^{(i)}_1|=w$ and $|B^{(i)}_2|=2^{r+1}$. Since $B^{(i)}_1$ is the smallest basket, there are no baskets of size $2^r$ in $\mathcal{B}_i$. Hence, after the merge, every remaining basket in $\mathcal{B}_i$ has size $2^{r+1}$, in which these baskets formed a chain by (P3). Removing $B_2^{(i)}$ from the beginning of this chain still leaves a chain, possibly empty. Again, the new basket $M$ is the unique exceptional wrapped basket, as shown in the proof of (P1). Thereforem the power-of-two baskets in $\mathcal{B}_{i+1}$ form the required chains, and $\mathcal{B}_{i+1}$ satisfies (P3).

    Finally, suppose that $|B^{(i)}_1|=2^r$ and $|B^{(i)}_2|=2^{r+1}$. In this case there is no wrapped basket in $\mathcal{B}_i$, and $B^{(i)}_1$ is the unique basket of size $2^r$. By (P3), the baskets of size $2^{r+1}$ followed by the baskets of size $2^r$ form a chain. Thus, the unique basket of size $2^r$ is the final block of this chain, while $B^{(i)}_2$ is the first block. Hence
    \[
        B^{(i)}_1=[m-2^{r-1},m-2^{r-1},\dots, m-1,m-1]
    \]
    and
    \[
        B^{(i)}_2=[0,0,1,1,\dots, 2^r-1,2^r-1].
    \]
    Their merge $M=B^{(i)}_1\uplus B^{(i)}_2$ is therefore the unique wrapped basket in $\mathcal{B}_{i+1}$. All other baskets have size $2^{r+1}$, and they formed a chain in $\mathcal{B}_i$. Removing $B^{(i)}_2$ from the beginning of this chain leaves a chain, possibly empty. Therefore, the power-of-two baskets in $\mathcal{B}_{i+1}$ satisfy the required chain condition, while $M$ is the unique exceptional wrapped basket. Thus $\mathcal{B}_{i+1}$ satisfies (P3) in this case as well.

Thus, $\mathcal{B}_{i+1}$ satisfies (P1), (P2), and (P3), completing the proof.\end{proof}

\begin{lemma}[Wrapped basket structure]\label{lem:wrapped-structure}
For $n\geq 1$ and $1\leq j\leq n$, let $X_j$ be the $j$th singleton basket in $\mathcal{B}_0$, that is, $X_j = \bigl[\lfloor \frac{j-1}{2}\rfloor\bigr]$. Suppose $\mathcal{B}_i$ has size-3-symmetry with respect to $2^r<w<2^{r+1}$, and suppose $W$ is the exceptional wrapped basket of size $w$, where $2^r<w<2^{r+1}$. Write $n=a2^r+s$ with $1\leq s<2^r$. Then $W$ is the multiset union of the final $2^r-s$ baskets in the list $X_1,\dots,X_n$ and the first $2s$ baskets in this list, i.e.\
\[
W=X_{n-(2^r-s)+1}\uplus \cdots \uplus X_n
\uplus X_1\uplus \cdots \uplus X_{2s}.
\]
In particular, $w=2^r+s$.
\end{lemma}

\begin{proof}
We prove the result by induction on $i$. If $\mathcal{B}_i$ contains no exceptional wrapped basket, there is nothing to prove. Thus, suppose that $\mathcal{B}_{i+1}$ contains an exceptional wrapped basket $W$.

If $W$ already belonged to $\mathcal{B}_i$ and was not one of the two baskets merged in passing from $\mathcal{B}_i$ to $\mathcal{B}_{i+1}$, then $W$ is unchanged. Moreover, the relevant value of $r$ is unchanged, so the conclusion follows from the induction hypothesis.

It remains to consider the cases in which the wrapped basket is created or changed by the merge. Let $B^{(i)}_1$ and $B^{(i)}_2$ be the two baskets merged in going from $\mathcal{B}_i$ to $\mathcal{B}_{i+1}$.

First, suppose that $\mathcal{B}_i$ has no exceptional wrapped basket, but $\mathcal{B}_{i+1}$ does. By Lemma \ref{lem:structural-props}, this can only happen when the last basket in the combined chain is merged with the first basket in the combined chain. More precisely, there exists an integer $u$ such that $B^{(i)}_1$ is the unique basket of size $2^u$, while $B^{(i)}_2$ has size $2^{u+1}$. Since $B^{(i)}_1$ is the unique basket of size $2^u$, (P3) implies that $B^{(i)}_1$ is the final basket in the combined chain, while $B^{(i)}_2$ is the first basket in the combined chain. Therefore, $B^{(i)}_1$ is the multiset union of the final $2^u$ singleton baskets in $\mathcal{B}_0$, and $B^{(i)}_2$ is the multiset union of the first $2^{u+1}$ singleton baskets in $\mathcal{B}_0$.

Let $W=B^{(i)}_1\uplus B^{(i)}_2$.
Then $W$ is wrapped and has size
$|W|=w=2^u+2^{u+1}=3\cdot 2^u.$
Thus, in $\mathcal{B}_{i+1}$, the basket $W$ is the exceptional wrapped basket whose size lies strictly between $2^{u+1}$ and $2^{u+2}$. Hence, for the statement of the lemma applied to $\mathcal{B}_{i+1}$, we set $2^r=2^{u+1}$.

Since $B^{(i)}_1$ is the unique basket of size $2^u$, $n$ is an odd multiple of $2^u$. Therefore, when we write $n=a2^r+s$
with $1\leq s<2^r$, we have $
s=2^u.$
Finally, since $W=B^{(i)}_1\uplus B^{(i)}_2$, the basket $W$ is exactly the multiset union of the final $2^u=2^r-s$ singleton baskets in $\mathcal{B}_0$ and the first $2^{u+1}=2s$ singleton baskets in $\mathcal{B}_0$. Hence $W$ is the multiset union of the final $2^r-s$ singleton baskets and the first $2s$ singleton baskets, as claimed.

Now suppose that $\mathcal{B}_i$ already contains an exceptional wrapped basket $W_i$. Suppose $W_i$ is exceptional with respect to the two power-of-two sizes $2^u$ and $2^{u+1}$. By the induction hypothesis, writing $n=a2^u+s$
with $1\leq s<2^u$, the basket $W_i$ is the multiset union of the final $2^u-s$ singleton baskets and the first $2s$ singleton baskets.

There are two ways in which $W_i$ can be changed by the next merge.

First, suppose that $B^{(i)}_1$ has size $2^u$ and $B^{(i)}_2=W_i$. By (P3), $B^{(i)}_1$ is the cyclically ordered basket of size $2^u$ immediately preceding $W_i$. Thus, merging $B^{(i)}_1$ into $W_i$ adds the $2^u$ singleton baskets immediately before the final $2^u-s$ singleton baskets already in $W_i$. Therefore, the new wrapped basket
\[
W=B^{(i)}_1\uplus W_i
\]
is the multiset union of the final $2^{u+1}-s$ singleton baskets and the first $2s$ singleton baskets.

The basket $W$ has size
\[
|W|=2^u+(2^u+s)=2^{u+1}+s,
\]
so in $\mathcal{B}_{i+1}$ it is exceptional with respect to the two power-of-two sizes $2^{u+1}$ and $2^{u+2}$. Thus, for $\mathcal{B}_{i+1}$, write $
2^r=2^{u+1}$. Since this merge occurs when there is a unique basket of size $2^u$, the integer $a$ in $n=a2^u+s$ is even. Hence, when we write $
n=a'2^r+s'$ with $1\leq s'<2^r$, we have $s'=s$. Therefore $W$ is the multiset union of the final $2^r-s'=2^{u+1}-s$
singleton baskets and the first $2s'=2s$
singleton baskets, as claimed.

Second, suppose that $B^{(i)}_1=W_i$ and $B^{(i)}_2$ has size $2^{u+1}$. By (P3), $B^{(i)}_2$ is the first basket of size $2^{u+1}$ immediately following $W_i$ in the combined chain. Thus, merging $W_i$ with $B^{(i)}_2$ adds the first $2^{u+1}$ singleton baskets immediately after the first $2s$ singleton baskets already in $W_i$. Therefore, the new wrapped basket
\[
W=W_i\uplus B^{(i)}_2
\]
is the multiset union of the final $2^u-s$ singleton baskets and the first $2s+2^{u+1}$ singleton baskets.

The basket $W$ has size
\[
|W|=(2^u+s)+2^{u+1}=3\cdot 2^u+s,
\]
so in $\mathcal{B}_{i+1}$ it is exceptional with respect to the two power-of-two sizes $2^{u+1}$ and $2^{u+2}$. Thus, for $\mathcal{B}_{i+1}$, write $
2^r=2^{u+1}.$
Since this merge occurs when there are no baskets of size $2^u$ remaining, the quotient $a$ in $n=a2^u+s$ is odd. Hence, when we write $
n=a'2^r+s'$ with $1\leq s'<2^r$, we have $
s'=2^u+s.$ Therefore $
2^r-s'=2^{u+1}-(2^u+s)=2^u-s$
and $2s'=2(2^u+s)=2^{u+1}+2s.$
So $W$ is the multiset union of the final $2^r-s'$ singleton baskets and the first $2s'$ singleton baskets, as claimed.

Since this covers every possible appearance or change of the exceptional wrapped basket, and since its size is $(2^r-s)+2s=2^r+s$, the proof is complete.
\end{proof}

We now show that the lengths of the baskets in the lex-merge strategy agree with the lexicographic order. To do so, it will be helpful to consider baskets of the following form. Let $q=2^{1/m}$.

When $a+h\leq m$, the basket represented by $I_h(a)$ is not wrapped and we have
\[
\ell(I_h(a))=2q^a(1+q+\cdots+q^{h-1})
=2q^a\left(\frac{q^h-1}{q-1}\right).
\]

\begin{lemma}\label{lem:lex-implies-length}
Let $\mathcal{B}_i$ be a collection appearing in the lex-merge strategy. If
$B,C\in \mathcal{B}_i$ and $B\preceq C$, then $\ell(B)\leq \ell(C)$.
\end{lemma}

\begin{proof}
By Lemma \ref{lem:structural-props}, we may assume that $B$ and $C$ are cyclically ordered and that $|B|,|C|\in \{2^r,w,2^{r+1}\}$, where $2^r < w < 2^{r+1}$. Moreover, at most one of $B$ or $C$ is wrapped, and if so, it has size $w$.

Since $\ell(I_h(a))=2q^a(1+q+\cdots+q^{h-1})$ whenever $a+h\leq m$, among non-wrapped cyclically ordered baskets with the same size, lexicographic order agrees with length order.

We may now assume that $|B|< |C|$. We first compare baskets of size $2^r$ with baskets of size $2^{r+1}$. Let $h=2^{r-1}$. Then a basket of size $2^r$ has $h$ doubled elements, while a basket of size $2^{r+1}$ has $2h$ doubled elements. Among non-wrapped cyclically ordered baskets of size $2^r$, the largest length is attained by \[B_0=I_h(m-h)=[(m-h)(m-h)\cdots (m-1)(m-1)].\] Among non-wrapped ordered baskets of size $2^{r+1}$, the smallest length is attained by \[C_0=I_{2h}(0)=[0,0,1,1,\dots,2h-1,2h-1].\] Thus it suffices to show that $\ell(B_0)\leq \ell(C_0)$.

Indeed we have \[\ell(B_0)=2q^{m-h}\left(\frac{q^h-1}{q-1}\right) \qquad \text{and} \qquad \ell(C_0)=2\left(\frac{q^{2h}-1}{q-1}\right)\] Therefore $\ell(B_0)\leq \ell(C_0)$ is equivalent to \[q^{m-h}(q^h-1)\leq q^{2h}-1.\] Since $q^m=2$, it remains to show $2(1-q^{-h})\leq q^{2h}-1$. Setting $y=q^h$, this becomes $2(1-y^{-1})\leq y^2-1$. Since $y>0$, this is equivalent to $0\leq y^3-3y+2=(y-1)^2(y+2)$, which is true. Hence $\ell(B)\leq \ell(C)$.

It remains to handle the cases where one of $B$ and $C$ is the exceptional wrapped basket. For $1\leq j\leq n$, let $
X_j=\left[\left\lfloor \frac{j-1}{2}\right\rfloor\right]$
be the $j$th singleton basket in the initial lex-merge collection $\mathcal{B}_0$. Let $W$ be the exceptional wrapped basket of size $w$. By Lemma \ref{lem:wrapped-structure}, if $2^r<w<2^{r+1}$ and $
n=a2^r+s$ with $1\leq s<2^r$, then $W$ is the multiset union of the final $2^r-s$ singleton baskets and the first $2s$ singleton baskets. That is,
\[
W=X_{n-(2^r-s)+1}\uplus\cdots\uplus X_n\uplus X_1\uplus\cdots\uplus X_{2s}.
\]
In particular, $w=2^r+s$.

First, suppose that $|B|=2^r$ and $C=W$. By (P3) in Lemma \ref{lem:structural-props}, the baskets of size $2^r$ form a chain. Hence, the largest length among baskets of size $2^r$ is attained by
\[
B_0=X_{n-(2^r-s)-2^r+1}\uplus\cdots\uplus X_{n-(2^r-s)}.
\]

Thus, it suffices to show that $\ell(B_0)\leq \ell(W)$.

Split $B_0$ as $B'_0\uplus B''_0$, where $B'_0$ contains the first $s$ singleton baskets of $B_0$ and $B''_0$ contains the remaining $2^r-s$ singleton baskets. That is,
\[
B'_0=X_{n-(2^r-s)-2^r+1}\uplus\cdots\uplus X_{n-(2^r-s)-2^r+s}
\]
and
\[
B''_0=X_{n-(2^r-s)-2^r+s+1}\uplus\cdots\uplus X_{n-(2^r-s)}.
\]
Similarly, split $W$ as $W'\uplus W''$, where
\[
W'=X_{n-(2^r-s)+1}\uplus\cdots\uplus X_n
\]
and
\[
W''=X_1\uplus\cdots\uplus X_{2s}.
\]
Then $|B''_0|=|W'|=2^r-s$, and $W'$ is lexicographically larger than $B''_0$. Hence $\ell(B''_0)\leq \ell(W')$. It remains to compare $B'_0$ and $W''$. Since $|B'_0|=s$ and $|W''|=2s$, and every singleton basket has length at least $1$ and less than $2$, we have
\[
\ell(B'_0)<2s\leq \ell(W'').
\]
Therefore $\ell(B_0)\leq \ell(W)$.

Now suppose that $B=W$ and $|C|=2^{r+1}$. By (P3) of Lemma \ref{lem:structural-props}, the baskets of size $2^{r+1}$ form a chain. Hence, the smallest length among baskets of size $2^{r+1}$ is attained by
\[
C_0=X_{2s+1}\uplus X_{2s+2}\uplus\cdots\uplus X_{2s+2^{r+1}}.
\]
Thus, it suffices to show that $\ell(W)\leq \ell(C_0)$.

Split $C_0$ as $C'_0\uplus C''_0$, where $C'_0$ contains the first $2s$ singleton baskets of $C_0$ and $C''_0$ contains the remaining $2^{r+1}-2s=2(2^r-s)$ singleton baskets. That is,
\[
C'_0=X_{2s+1}\uplus\cdots\uplus X_{4s}
\]
and
\[
C''_0=X_{4s+1}\uplus\cdots\uplus X_{2s+2^{r+1}}.
\]
Similarly, split $W$ as $W'\uplus W''$, where
\[
W'=X_{n-(2^r-s)+1}\uplus\cdots\uplus X_n
\]
and
\[
W''=X_1\uplus\cdots\uplus X_{2s}.
\]
Then $|W''|=|C'_0|=2s$, and $C'_0$ is lexicographically larger than $W''$. Hence $\ell(W'')\leq \ell(C'_0)$. It remains to compare $W'$ and $C''_0$. Since $|W'|=2^r-s$ and $|C''_0|=2^{r+1}-2s=2(2^r-s)$, and every singleton basket has length at least $1$ and less than $2$, we have
\[
\ell(W')<2(2^r-s)\leq \ell(C''_0).
\]
Therefore $\ell(W)\leq \ell(C_0)$.

Combining the cases, whenever $B\preceq C$ in $\mathcal{B}_i$, we have $\ell(B)\leq \ell(C)$. \end{proof}

\subsection{Proof of the upper-bound of \cref{thm:lex-merge}}\label{sec:upper}

To prove \cref{thm:lex-merge} it suffices to show that \[\disc(\text{LM}_n) = 2^{1-\frac{1}{m}}.\]

To do so, we first prove a key lemma showing that the discrepancies of
$\mathcal{B}_0,\dots,\mathcal{B}_{n-1}$ in the lex-merge strategy of order $n$
are weakly decreasing. The proof of \cref{thm:lex-merge} will then follow easily.

\begin{lemma}\label{lem:merge-monotonicity}
Fix $n\geq 1$, and let
$\mathcal{B}_0,\mathcal{B}_1,\dots,\mathcal{B}_{n-1}$
be the collections of baskets produced by the lex-merge strategy of order $n$. Then, for every $0\leq i\leq n-2$,
\[
\disc(\mathcal{B}_{i+1})\leq \disc(\mathcal{B}_i).
\]
\end{lemma}

\begin{proof}
Let $\mathcal{B}_i=\{B^{(i)}_1,\dots,B^{(i)}_t\}$, where $B^{(i)}_1\preceq \cdots \preceq B^{(i)}_t$, and let $M=B^{(i)}_1\uplus B^{(i)}_2$ be the basket formed by the merge from $\mathcal{B}_i$ to $\mathcal{B}_{i+1}$. Thus, $\mathcal{B}_{i+1}=\{B^{(i)}_3,\dots,B^{(i)}_t,M\}$.

By Lemma \ref{lem:lex-implies-length}, the lexicographic order on $\mathcal{B}_i$ agrees with length order, so
\[
\ell(B^{(i)}_1)\leq \ell(B^{(i)}_2)\leq \cdots \leq \ell(B^{(i)}_t).
\]
In particular,
\[
\disc(\mathcal{B}_i)=\frac{\ell(B^{(i)}_t)}{\ell(B^{(i)}_1)}.
\]

If $t=2$, then $\mathcal{B}_{i+1}$ consists of a single basket, so $\disc(\mathcal{B}_{i+1})=1\leq \disc(\mathcal{B}_i)$. We may therefore assume $t\geq 3$.

All baskets of $\mathcal{B}_{i+1}$ except $M$ already belonged to $\mathcal{B}_i$. Thus, if $M$ is not the largest basket lexicographically in $\mathcal{B}_{i+1}$, then by Lemma \ref{lem:lex-implies-length}, the maximum length has not increased. Also, the minimum length has not decreased: the shortest basket $B^{(i)}_1$ has been removed, and $\ell(M)=\ell(B^{(i)}_1)+\ell(B^{(i)}_2)>\ell(B^{(i)}_1)$. Therefore, in this case, $\disc(\mathcal{B}_{i+1})\leq \disc(\mathcal{B}_i)$.

It remains to consider the case where $M$ is the largest basket lexicographically in $\mathcal{B}_{i+1}$. Then the smallest basket in $\mathcal{B}_{i+1}$ is $B^{(i)}_3$, and by Lemma \ref{lem:lex-implies-length}, we have
\[
\disc(\mathcal{B}_{i+1})
=
\frac{\ell(M)}{\ell(B^{(i)}_3)}
=
\frac{\ell(B^{(i)}_1)+\ell(B^{(i)}_2)}{\ell(B^{(i)}_3)}.
\]
So it suffices to show that
\[
\frac{\ell(B^{(i)}_1)+\ell(B^{(i)}_2)}{\ell(B^{(i)}_3)}
\leq
\frac{\ell(B^{(i)}_t)}{\ell(B^{(i)}_1)}.
\]

By Lemma \ref{lem:structural-props} we may assume that $\mathcal{B}_i$ has size-3-symmetry with respect to $2^r<w<2^{r+1}$. First, suppose $|B^{(i)}_1|=|B^{(i)}_2|$. We may assume that $|B^{(i)}_1|=|B^{(i)}_2|=2^r$, and let $h=2^{r-1}$. Then $B^{(i)}_1$ and $B^{(i)}_2$ are consecutive baskets in the size-$2^r$ chain, and $M$ is the corresponding basket of size $2^{r+1}$. If $M$ is the longest basket after the merge, then $B^{(i)}_t$ is the size-$2^{r+1}$ basket immediately preceding $B^{(i)}_1$ in the chain. Thus, we have
\[
B^{(i)}_1=I_h(a),\qquad
B^{(i)}_2=I_h(a+h),\qquad
M=I_{2h}(a),
\]
and
\[
\ell(B^{(i)}_3)\geq \ell(I_h(a+2h))
\]
where equality holds if there is another size-$2^r$ basket remaining, and otherwise $B^{(i)}_3$ is even larger. Also,
\[
B^{(i)}_t=I_{2h}(a-h).
\]
 Therefore, it is enough to check
\[
\frac{\ell(I_{2h}(a))}{\ell(I_h(a+2h))}
\leq
\frac{\ell(I_{2h}(a-h))}{\ell(I_h(a))}.
\]
Using $\ell(I_h(a))=2q^a(1+\cdots+q^{h-1})$, the inequality reduces to $
q^{a}q^{a}\leq q^{a-h}q^{a+2h}$,
which is certainly true. Thus, the discrepancy does not increase in this case.

Now suppose $|B^{(i)}_1|<|B^{(i)}_2|$. There are three cases for $(|B^{(i)}_1|,|B^{(i)}_2|)$:
\[
(2^r,w),\qquad (w,2^{r+1}),\qquad (2^r,2^{r+1}).
\]

The first two cases involve the exceptional wrapped basket. By Lemma \ref{lem:wrapped-structure}, the wrapped basket consists of the final $2^r-s$ singleton baskets in $\mathcal{B}_0$ together with the first $2s$ singleton baskets in $\mathcal{B}_0$, where $n=a2^r+s$ and $1\leq s<2^r$.

It suffices to prove the following cases when $n$ is even. Indeed, when $n$ is odd, the only change is that the final element $m-1$ appears once rather than twice. In each case below, this only decreases the length of $M$, so the left-hand side of
\[
\ell(B^{(i)}_1)\ell(M)\leq \ell(B^{(i)}_3)\ell(B^{(i)}_t)
\]
only becomes smaller. We now consider the three cases separately.

\textbf{Case 1: $(|B_1^{(i)}|,|B_2^{(i)}|) = (2^r, w)$.} Let $h=2^{r-1}$. Since $n$ is even, we can write $s=2k$, so $1\leq k<h$. By Lemma \ref{lem:wrapped-structure}, the exceptional wrapped basket has the form
\[
B^{(i)}_2=I_{h-k}(m-h+k)\uplus I_{2k}(0). \qquad\implies\qquad \ell(B_2^{(i)})=2q^{m-h+k}\left(\frac{q^{h-k}-1}{q-1}\right)+2\left(\frac{q^{2k}-1}{q-1}\right)
\]
Since $B^{(i)}_1$ is the unique basket of size $2^r$, we have
\[
B^{(i)}_1=I_h(m-2h+k) \qquad\implies\qquad \ell(B_1^{(i)})=2q^{m-2h+k}\left(\frac{q^{h}-1}{q-1}\right).
\]
Moreover,
\[
M=B^{(i)}_1\uplus B^{(i)}_2
=
I_{2h-k}(m-2h+k)\uplus I_{2k}(0) \qquad\implies\qquad \ell(M) = 2q^{m-2h+k}\left(\frac{q^{2h-k}-1}{q-1}\right) +2\left(\frac{q^{2k}-1}{q-1}\right).
\]
The next smallest basket is the first basket of size $2^{r+1}$ after the wrapped basket, so
\[
B^{(i)}_3=I_{2h}(2k) \qquad\implies\qquad \ell(B_3^{(i)})=2q^{2k}\left(\frac{q^{2h}-1}{q-1}\right),
\]
and the largest basket in $\mathcal{B}_i$ is the final basket of size $2^{r+1}$ before $B^{(i)}_1$, namely
\[
B^{(i)}_t=I_{2h}(m-4h+k) \qquad\implies\qquad \ell(B_t^{(i)})=2q^{m-4h+k}\left(\frac{q^{2h}-1}{q-1}\right).
\]

We need to prove $
\ell(B^{(i)}_1)\ell(M)\leq \ell(B^{(i)}_3)\ell(B^{(i)}_t)$. Using that $q^m=2$, the desired inequality is equivalent to
\[
q^{2h}(1+q^{2k}-2q^{k-2h})
\leq
q^{2k}(q^h-1)(q^h+1)^2,
\]
or, equivalently,
\[
0\leq q^{2k+3h}-q^{2k+h}-q^{2k}+2q^k-q^{2h}.
\]

Set $x=q^k$ and $y=q^h$. Since $1\leq k<h$ and $q>1$, we have $1\leq x<y$. The desired inequality becomes
\[
0\leq x^2y^3-x^2y-x^2+2x-y^2.
\]
Define $
F(x,y)=x^2y^3-x^2y-x^2+2x-y^2.$
We show that $F(x,y)\geq 0$ whenever $1\leq x\leq y$. Then
\[
\frac{\partial F}{\partial x}=2x(y^3-y-1)+2.
\]
If $y^3-y-1\geq 0$, then $\frac{\partial F}{\partial x}>0$. If $y^3-y-1<0$, then $\frac{\partial F}{\partial x}$ is minimized on $1\leq x\leq y$ when $x=y$, and at $x=y$ we have
\[
2y(y^3-y-1)+2=2(y^4-y^2-y+1)
=2(y-1)(y^3+y^2-1)>0.
\]
Thus, $F$ is increasing in $x$ on $1\leq x\leq y$, and so
\[
F(x,y)\geq F(1,y)=y^3-y^2-y+1=(y-1)^2(y+1)\geq 0.
\]

and so $\ell(B^{(i)}_1)\ell(M)\leq \ell(B^{(i)}_3)\ell(B^{(i)}_t)$.

\textbf{Case 2: $(|B_1^{(i)}|,|B_2^{(i)}|)=(w,2^{r+1})$.}
Let $h=2^{r-1}$. Since $n$ is even, we can write $s=2k$, so $1\leq k<h$. By Lemma \ref{lem:wrapped-structure}, the exceptional wrapped basket has the form
\[
B^{(i)}_1=I_{h-k}(m-h+k)\uplus I_{2k}(0). \qquad\implies\qquad \ell(B_1^{(i)})=2q^{m-h+k}\left(\frac{q^{h-k}-1}{q-1}\right)+2\left(\frac{q^{2k}-1}{q-1}\right).
\]
Since there are no baskets of size $2^r$, the basket $B^{(i)}_2$ is the first basket of size $2^{r+1}$ after the wrapped basket. Thus,
\[
B^{(i)}_2=I_{2h}(2k) \qquad\implies\qquad \ell(B_2^{(i)})=2q^{2k}\left(\frac{q^{2h}-1}{q-1}\right).
\]
Moreover,
\[
M=B^{(i)}_1\uplus B^{(i)}_2
=
I_{h-k}(m-h+k)\uplus I_{2k+2h}(0) \qquad\implies\qquad \ell(M)=2q^{m-h+k}\left(\frac{q^{h-k}-1}{q-1}\right)+2\left(\frac{q^{2k+2h}-1}{q-1}\right).
\]
The next smallest basket is the next basket of size $2^{r+1}$ after $B^{(i)}_2$, so
\[
B^{(i)}_3=I_{2h}(2k+2h) \qquad\implies\qquad \ell(B_3^{(i)})=2q^{2k+2h}\left(\frac{q^{2h}-1}{q-1}\right),
\]
and the largest basket in $\mathcal{B}_i$ is the final basket of size $2^{r+1}$ before the wrapped basket, namely
\[
B^{(i)}_t=I_{2h}(m-3h+k) \qquad\implies\qquad \ell(B_t^{(i)})=2q^{m-3h+k}\left(\frac{q^{2h}-1}{q-1}\right).
\]

We need to prove $
\ell(B^{(i)}_1)\ell(M)\leq \ell(B^{(i)}_3)\ell(B^{(i)}_t).$

Using that $q^m=2$, the desired inequality is equivalent to
\[
\left(1+q^{2k}-2q^{k-h}\right)\left(1+q^{2k+2h}-2q^{k-h}\right)
\leq
2q^{3k-h}(q^{2h}-1)^2.
\]
Set $x=q^k$ and $z=q^{h-k}$. Then $x\geq 1$, $z>1$, and this inequality becomes
\[
\left(1+x^2-\frac{2}{z}\right)\left(1+x^4z^2-\frac{2}{z}\right)
\leq
\frac{2x^2}{z}(x^2z^2-1)^2.
\]
After moving the left-hand side to the right, the difference factors as
\[
\frac{x^2z-1}{z^2}
\left(2x^4z^4-x^4z^3+x^2z^3-3x^2z^2+z^2-4z+4\right).
\]
The first factor is positive. For the second factor, set $X=x^2$ and define
\[
G(X,z)=2X^2z^4-X^2z^3+Xz^3-3Xz^2+z^2-4z+4.
\]
Since
\[
\frac{\partial G}{\partial X}=z^2\bigl(4Xz^2-2Xz+z-3\bigr),
\]
and $X\geq 1$, we have
\[
\frac{\partial G}{\partial X}\geq z^2(4z^2-z-3)=z^2(z-1)(4z+3)>0.
\]
Thus, $G$ is increasing in $X$, so
\[
G(X,z)\geq G(1,z)=2(z-1)^2(z^2+2z+2)\geq 0.
\]
Therefore, the desired inequality holds, and so $
\ell(B^{(i)}_1)\ell(M)\leq \ell(B^{(i)}_3)\ell(B^{(i)}_t)$.

\textbf{Case 3: $(|B_1^{(i)}|,|B_2^{(i)}|)=(2^r,2^{r+1})$.}
Let $h=2^{r-1}$. In this case, there is no wrapped basket in $\mathcal{B}_i$. Moreover, $B^{(i)}_1$ is the unique basket of size $2^r$. Thus, $B^{(i)}_1$ is the final basket in the combined chain, while $B^{(i)}_2$ is the first basket in the combined chain. Hence,

\[
B^{(i)}_1=I_h(m-h),\qquad
B^{(i)}_2=I_{2h}(0),
\qquad\Longrightarrow\qquad
M=I_h(m-h)\uplus I_{2h}(0).
\]

The next smallest basket is the next basket of size $2^{r+1}$, so
\[
B^{(i)}_3=I_{2h}(2h),
\]
and the largest basket in $\mathcal{B}_i$ is the final basket of size $2^{r+1}$ before $B^{(i)}_1$, namely
\[
B^{(i)}_t=I_{2h}(m-3h).
\]

We need to prove that
\[
\ell(B^{(i)}_1)\ell(M)\leq \ell(B^{(i)}_3)\ell(B^{(i)}_t).
\]
Using the formula for $\ell(I_k(a))$ and $q^m=2$, the desired inequality is equivalent to
\[
2q^{-h}+(q^h+1)\leq (q^h+1)^2.
\]
After rearranging, this becomes
\[
0\leq q^{2h}+q^h-2q^{-h}.
\]
Multiplying by $q^h>0$, this is equivalent to
\[
0\leq q^{3h}+q^{2h}-2=(q^h-1)(q^{2h}+2q^h+2),
\]
which is true. Thus $\ell(B^{(i)}_1)\ell(M)\leq \ell(B^{(i)}_3)\ell(B^{(i)}_t)$.

Therefore, in all cases, $
\disc(\mathcal{B}_{i+1})\leq \disc(\mathcal{B}_i)$.
\end{proof}

\begin{theorem}\label{conj:lex-merge}
For any $n\geq 1$,
\[
\disc(\mathrm{LM}_n)=2^{1-\frac1m},
\]
where $m=\left\lceil \frac n2\right\rceil$.
\end{theorem}

\begin{proof}
Let $q=2^{1/m}$.
By Lemma \ref{lem:merge-monotonicity}, the discrepancies along the lex-merge strategy are weakly decreasing:
\[
\disc(\mathcal{B}_0)\geq \disc(\mathcal{B}_1)\geq \cdots \geq \disc(\mathcal{B}_{n-1}).
\]

Thus, the maximum discrepancy over the whole lex-merge strategy occurs at the initial collection $\mathcal{B}_0$. It remains only to compute $\disc(\mathcal{B}_0)$.

By definition, $\mathcal{B}_0$ consists of singleton baskets. The smallest singleton lexicographically is $[0]$, whose length is $\ell([0])=q^0=1$, and the largest singleton lexicographically is $[m-1]$, whose length is $\ell([m-1])=q^{m-1}$. By Lemma \ref{lem:lex-implies-length}, these are respectively the minimum- and maximum-length baskets in $\mathcal{B}_0$. Therefore
\[
\disc(\mathcal{B}_0)=\frac{q^{m-1}}{1}=q^{m-1}.
\]
Since $q=2^{1/m}$, this becomes
\[
q^{m-1}=\bigl(2^{1/m}\bigr)^{m-1}=2^{1-\frac1m}.
\]

as desired.
\end{proof}

\section{Concluding remarks}\label{sec:conclusion}

We wonder if the lex-merge strategy the unique strategy of length $n$ achieving discrepancy $2^{1-1/\lceil n/2\rceil}$.
It is easy to argue that if $\mcal S=(\mcal I_1,\dots,\mcal I_n)$ is a strategy with $\disc(\mcal S)<2$, then, at each time $t$, we must divide the largest interval in $\mcal I_t$, and that one of the two resulting intervals must be the smallest interval in $\mcal I_{t+1}$.
If one could show that \emph{both} resulting intervals were the smallest intervals in $\mcal I_{t+1}$ assuming $\disc(\mcal S)=\disc(n)$, then this would essentially prove the uniqueness of lex-merge.
\medskip

Next, we remark on a difference between de Bruijn--Erd\H{o}s' problem and ours that was hinted at in the introduction.
De Bruijn and Erd\H{o}s were actually concerned with infinite-length strategies instead of those finite-length strategies considered in this paper.
For an infinite-length strategy $\mcal S=(\mcal I_1,\mcal I_2,\dots)$, define
\[
    \disc_n(\mcal S)\eqdef\max_{t\leq n}\ \disc(\mcal I_t).
\]
Naturally, $\disc_n(\mcal S)\geq\disc(n)=2^{1-1/\lceil n/2\rceil}=2-{4\ln 2\over n}+O\bigl(\frac 1{n^2}\bigr)$ for any infinite-length strategy $\mcal S$ and any positive integer $n$, but the reverse inequality may not hold.
De Bruijn and Erd\H{o}s actually showed that there exists an infinite-length strategy $\mcal S$ for which
\[
    \disc_n(\mcal S)={\log(1+1/n)\over \log\bigl((1-{1\over 2n})^{-1}\bigr)}=2-{3\over 2n}+O\biggl({1\over n^2}\biggr).
\]
It is natural to wonder if the first-order term of ${3\over 2n}$ here is tight.
\begin{question}\label[question]{infinite}
    If $\mcal S$ is an infinite-length strategy, then must
    \[
        \disc_n(\mcal S)\geq 2-{3\over 2n}-O\biggl({1\over n^2}\biggr)?
    \]
\end{question}
We suspect that the lower-bound of $\disc(n)$ is not tight given that the lex-merge strategy seems intimately tied to the fact that the strategy has finite length.
Of course, this suspicion  is certainly true if \Cref{infinite} has a positive answer.
Regardless to the truth of \Cref{infinite}, we confidently conjecture the following weaker statement which would show a separation between finite- and infinite-length strategies:
\begin{conjecture}
    For any infinite-length strategy $\mcal S$,
    \[
        \liminf_{n\to\infty}{\disc_n(\mcal S)-\disc(n)\over 1/n}>0.
    \]
\end{conjecture}


\end{document}